\documentclass[final]{siamltex}
\usepackage{epsfig,psfrag,epsf}
\usepackage{amsmath,amsfonts}

\usepackage{tikz}
\usetikzlibrary{arrows,shapes}
\usetikzlibrary{fadings}
\usetikzlibrary{intersections}
\usetikzlibrary{decorations.pathreplacing}

\DeclareMathOperator{\re}{\text{Re}}
\DeclareMathOperator{\im}{\text{Im}}

\newcommand{\INF}{{\infty}}

\newcommand{\tta}{\theta}

\newcommand{\OM}{\Omega}

\newcommand{\sph}{{{\mathbf S}^ 1}}

\newcommand{\del}{\partial}

\newcommand{\ol}{\overline}
\newcommand{\ds}{\displaystyle}

\newcommand{\Gam}{\varGamma}
\newcommand{\BR}{\mathbb{R}}

\newcommand{\fii}{{\varphi}}
\newcommand{\bu}{{\bf u}}
\newcommand{\bv}{{\bf v}}

\newcommand{\bg}{{\bf g}}

\newcommand{\bF}{{\bf F}}

\newcommand{\LL}{{\mathcal L}}
\newcommand{\B}{\mathcal{B}}

\newcommand{\HT}{\mathcal{H}}

\title{On the range characterization of the two dimensional attenuated Doppler transform
\thanks{Received by the editors 2014.}}

\author{Kamran Sadiq\thanks{Department of Mathematics, University
        of Central Florida, Orlando, FL, USA ({\tt ksadiq@knights.ucf.edu}).}\and
        Alexandru Tamasan\thanks{Department of Mathematics, University
        of Central Florida, Orlando, FL, USA ({\tt tamasan@math.ucf.edu}).}}

\begin{document}

\maketitle

\begin{abstract}
We characterize the range of the attenuated and non-attenuated $X$-ray transform of compactly supported vector fields in the plane. The characterization is in terms of a Hilbert transform associated with the $A$-analytic functions \`{a} la Bukhgeim. As an application we determine necessary and sufficient conditions for the attenuated Doppler and $X$-ray data to be mistaken for each other.
\end{abstract}


\pagestyle{myheadings} \thispagestyle{plain} \markboth{K. Sadiq and  A. Tamasan}{
Range characterization of the attenuated Doppler transform}

\section{Introduction} \label{S:intro}

Necessary and sufficient constraints on the range of the (non-attenuated) Radon transform of zero order tensors in the Euclidean space have been known since the works in \cite{gelfandGraev}, \cite{helgason}, and \cite{ludwig}. In the case of the attenuated Radon transform with constant attenuation some range conditions can be inferred from \cite{kuchmentLvin}, \cite{aguilarKuchment} and \cite{aguilarEhrenpreisKuchment}. For a varying attenuation, range constraints for the two dimensional $X$-ray transform were given in \cite{novikov02} based on the inversion method in \cite{novikov01}, see also \cite{bal}. A separate method to invert the two dimensional attenuated $X$-ray transform based on the theory of $A$-analytic functions was originally developed in \cite{ABK}. In \cite{sadiqtamasan01} the authors introduce a Hilbert transform corresponding to $A$-analytic maps, and use it to characterize the range of the attenuated Radon transform of compactly supported functions.

The problem of inversion of the $X$-ray transform of higher order tensor fields have been formulated in \cite{vladimirBook} in the geometric setting of Riemannian manifolds with boundary; see \cite{SSLP} for the Euclidean setting. Coming from the practical procedure of acquiring data, the $X$-ray transform of vector fields is also known as the Doppler transform. Various partial results (e.g., \cite{sharafutdinov}, \cite{sharafutdinovSkokanUhlmann}) culminated with the inversion formulas for recovering the solenoidal part of one-tensors on simple Riemannian surfaces with boundary \cite{pestovUhlmann04}; see also \cite{pfitzenreiterSchuster}. Injectivity in the attenuated case for both 0- and 1-tensors is much more recent \cite{saloUhlmann11}; see also \cite{holmanStefanov} for a more general weighted transform. Inversion formulas to tensors of higher orders have been found in \cite{kazantsevBukhgeimJr06} for the Euclidean case, and \cite{paternainSaloUhlmann13} for the Riemannian case. However, these works do not address range characterization.

The first range characterizations of the (non-attenuated) $X$-ray transform is given in \cite{pestovUhlmann04} for both 0- and 1-tensors supported
on simple Riemannian surfaces with boundary. This characterization is given in terms of the scattering relation.

We consider here the problem  of the range characterization of both attenuated  and non-attenuated Doppler transform in a strictly convex bounded
domain in the Euclidian plane.
Our approach uses the Hilbert transform for $A$-analytic maps in \cite{sadiqtamasan01} and relies on new identities enjoyed by such maps,
see Lemma \ref{UconjLemma}. The characterization concerns the Fourier modes (in the angular variables) of the data and can be interpreted as a
characterization of the scattering relation in \cite{pestovUhlmann04} for the Euclidean case.
Of particular interest in the non-attenuating case, we found that the even and odd negative Fourier modes play a different role. More precisely, the odd negative index modes are doubly constrained, and that the zero-th order Fourier mode \emph{is independent} of the other modes; see Theorem \ref{NADopplerT}. This is a direct equivalence to the non-uniqueness in the inversion (up to a gradient of a compactly supported map).
In the positive attenuation case, the Doppler transform is uniquely invertible as shown in  \cite{kazantsevBukhgeimJr06}, see also \cite{tamasan07}. In such a case the zero-th Fourier mode is uniquely determined by the negative modes of the boundary data, see Theorem \ref{ADopplerT}.

The method used in the characterization will explain when (and only then) the attenuated X-ray and Doppler data can be confounded for each other,
see Section \ref{xraydopplermistaken}. Practical applications may include noise reduction and missing data completion in medical imaging methods
such as Single Photon, Positron Emission Computed Tomography, or Doppler Tomography \cite{nattererWubbeling}.

Let $\OM \subset \BR^2$ be a bounded strictly convex domain in the plane and $\Gam$ be its boundary. The unit sphere is denoted by $\sph$.
For any $(x,\theta)\in\ol\OM\times \sph$, let $\tau_\pm(x,\theta)$ denote the distance from $x$ in the $\pm\theta$ direction to the boundary, and distinguish the endpoints $x^\pm_\theta\in\Gam$ of the chord in the direction of $\theta$ passing through $x$ by
\begin{align}\label{xthetapm}
x^\pm_\theta:=x\pm\tau_\pm(x,\theta)\theta,
\end{align}as in Figure \ref{fig:1} below. Note that $\tau(x,\tta)=\tau_+(x,\tta)+\tau_-(x,\tta)$ is the length of the cord.

\begin{figure}[ht]
\centering
\begin{tikzpicture}[scale=1.5,cap=round,>=latex]

 \draw[thick] (0cm,0cm) circle(1cm);
 \draw[gray] (0cm,0cm) -- (45:1cm);
 \filldraw[black] (45:1cm) circle(1.2pt);
  \coordinate[label=right:$x_{\theta}^{+}$] (x_{+}) at (45:1.01cm);

 \draw[->] (45:1.2cm) -- (45:1.6cm);
 \coordinate[label=left:$\theta$] (theta1) at (45:1.4cm);

 \filldraw[black] (45:0.2cm) circle(1.2pt);
  \coordinate[label=right:$x$] (x) at(45:0.2cm);
  \coordinate[label=above:$\OM$] (OM) at(120:0.6cm);

  \draw[gray] (0cm,0cm) -- (225:1cm);
  \filldraw[black] (225:1cm) circle(1.2pt);

  \coordinate[label=below:$x_{\theta}^{-}$] (x_{-}) at (225:1cm);

  \draw[->] (225:1.6cm) -- (225:1.2cm);
  \coordinate[label=left:$\theta$] (theta2) at (225:1.2cm);

  \tikzset{
    position label/.style={
       above = 3pt,
       text height = 1.5ex,
       text depth = 1ex
    },
   brace/.style={
     decoration={brace,mirror},
     decorate
   }
}
\draw [brace,decoration={raise=0.5ex}] [brace]  (x_{+}.north) -- (x.north) node [position label, pos=0.6, rotate = 45,scale=0.8] {$\tau_{+}(x,\tta)$};
 \tikzset{
    position label/.style={
       below = 3pt,
       text height = 1.5ex,
       text depth = 1ex
    },
   brace/.style={
     decoration={brace,mirror},
     decorate
   }
}

 \draw [brace,decoration={raise=0.5ex}] (x_{-}.north) -- (x.north) node [position label, pos=0.55, rotate = 45, scale=0.8] {$\tau_{-}(x,\tta)$};
\end{tikzpicture}
\caption{Definition of $\tau_\pm(x,\tta)$ } \label{fig:1}
\end{figure}
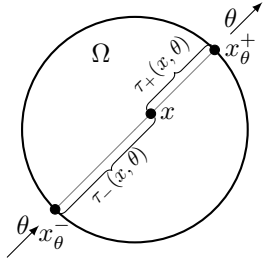

We consider a real valued function $a \in C^1_0(\ol \OM)$. For each $x \in \OM$ and $\tta = (\cos \fii , \sin \fii) \in \sph$ the {\em divergence beam transform} of $a$ is
\begin{align}\label{divbeam}
Da(x,\tta): =\int_{0}^{\tau(x,\tta)} a(x+t\tta)dt.
\end{align}

The {attenuated $X$-ray transform} (with attenuation $a$) of some function $f\in L^1(\OM)$   is given by
\begin{align}\label{AttRT}
\int_{-\tau_-(x,\tta)}^{\tau_+(x,\tta)} f(x+t\theta)e^{-Da(x+t\theta,\theta)}dt,
\end{align} and the attenuated Doppler Transform of some vector field $\bF\in L^1(\OM;\BR^2)$ is given by
\begin{align}\label{DopplerT}
\int_{-\tau_-(x,\tta)}^{\tau_+(x,\tta)} ( \tta \cdot \bF)(x+t\tta)e^{-Da(x+t\theta,\theta)}dt .
\end{align}

Both transforms are functions on the tangent bundle of the circle, however we will describe the constraints in terms of a function $g$ on $\Gam\times\sph$ as follows:
\begin{definition}
(i) We say that $g$ is an {\em attenuated $X$-ray transform } of $f$ with attenuation $a$, if
\begin{align}\label{Radon_definition}
g(x^+_\theta,\theta)-\left[e^{-Da} g\right](x^-_\theta,\theta)= \int_{\tau_-(x,\theta)}^{\tau_+(x,\theta)} f(x+t\theta)e^{-Da(x+t\theta,\theta)}dt,
\end{align}for  $(x,\theta)\in\ol\OM\times \sph$. We use the notation $g\in R_af$, and $g\in Rf$ if $a \equiv 0$.

(ii) We say that $g$ is
 an {\em attenuated Doppler transform }of $\bF$ with attenuation $a$, if
\begin{align}\label{Doppler_definition}
g(x^+_\theta,\theta)-\left[e^{-Da} g\right](x^-_\theta,\theta)= \int_{\tau_-(x,\theta)}^{\tau_+(x,\theta)} ( \tta \cdot \bF)(x+t\tta)e^{-Da(x+t\theta,\theta)}dt ,
\end{align}for  $(x,\theta)\in\ol\OM\times \sph$. We use the notation $g\in D_a\bF$, and $g\in D\bF$ if $a \equiv 0$.
\end{definition}

These definitions are motivated by the connection with the transport model as follows. If $v$ is a solution to
\begin{equation} \label{AttenRadonTEq}
    \tta\cdot\nabla v(x,\tta) +a(x) v(x,\theta) = f(x), \quad  (x,\tta)\in \OM \times \sph,
\end{equation}
then $g:=v|_{\Gamma\times\sph}$ satisfies \eqref{Radon_definition}, and similarly, if $u$ is a solution to
\begin{equation} \label{DopplerTransEq}
    \tta\cdot\nabla u(x,\tta) +a(x) u(x,\theta) =  \tta \cdot \bF (x), \quad  (x,\tta)\in \OM \times \sph,
\end{equation}
then $g:=u|_{\Gamma\times\sph}$ satisfies \eqref{Doppler_definition}. Let $$\Gamma_\pm =\{(x,\tta)\in \Gamma\times \sph: \pm\tta\cdot \nu(x)>0\},$$
where $\nu(x)$ is the outer normal. It is clear that specifying $v$, or $u$ on $\Gamma_-$ well defines a unique solution to the transport equations \eqref{AttenRadonTEq}, respectively \eqref{Doppler_definition}. In particular, those traces $g$ vanishing on ${\Gamma_-}$, make $g|_{\Gamma_+}$ coincide with the classical definitions \eqref{AttRT}, respectively \eqref{DopplerT}.

\section{Basic properties of $A$-analytic maps}
In this section we briefly introduce the properties of $A$-analytic maps needed later, and introduce notation.
The presentation follows mainly from \cite{sadiqtamasan01}. Only the new results are proven.

For $z=x_1+ix_2$, we consider the Cauchy-Riemann operators
\begin{align} \label{CauchyRiemannOp}
\ol{\del} = \left( \del_{x_{1}}+i \del_{x_{2}} \right) /2 ,\quad \del = \left( \del_{x_{1}}- i \del_{x_{2}} \right) /2
\end{align}

Let $l_\INF (,l_1)$ be the space of bounded (, respectively summable) sequences,
$\mathcal{L}:l_\infty\to l_\infty$ be the left shift
\begin{align*}
\mathcal{L}  \langle u_{-1}, u_{-2}, ... \rangle =  \langle u_{-2}, u_{-3}, u_{-4}, ... \rangle,
\end{align*}and $\mathcal{L}^{k}=\underbrace{\mathcal{L}\circ \cdots \circ \mathcal{L}}_{k}$ be its $k$-th composition;
we will need only $k =1,2$.
\begin{definition}
A sequence valued map
$$z\mapsto  \bu(z): = \langle u_{-1}(z),u_{-2}(z),u_{-3}(z)... \rangle$$
is called {\em $\mathcal{L}^k$-analytic}, $k=1,2$, if $\bu\in C(\ol\OM;l_\INF)\cap C^1(\OM;l_\INF)$ and
\begin{equation}\label{Aanalytic}
\ol{\del} \bu (z) + \mathcal{L}^k \del \bu (z) = 0,\quad z\in\OM.
\end{equation}
\end{definition}

For $0<\alpha<1$ and $k=1,2$, we recall the Banach spaces in \cite{sadiqtamasan01}:
\begin{equation} \label{lGamdefn}
 l^{1,k}_{\INF}(\Gam): = \left \{ \bu=  \langle u_{-1}, u_{-2}, ... \rangle  : \sup_{\zeta\in \Gam}\sum_{j=1}^{\INF}  j^{k} \lvert u_{-j}(\zeta) \rvert < \INF \right \},
\end{equation}
\begin{equation} \label{CepsGamdefn}
 C^{\alpha}(\Gam ; l_1) := \left \{ \bu:
\sup_{\xi\in \Gam} \lVert \bu(\xi)\rVert_{\ds l_{1}} + \underset{{\substack{
            \xi,\eta \in \Gam \\
            \xi\neq \eta } }}{\sup}
 \frac{\lVert \bu(\xi) - \bu(\eta)\rVert_{\ds l_{1}}}{|\xi - \eta|^{ \alpha}} < \INF \right \}.
\end{equation}
By replacing $\Gam$ with $\overline\OM$ and $l_{1}$ with $l_{\INF}$ in \eqref{CepsGamdefn} we similarly define  $C^{\alpha}(\overline\OM ; l_1)$, respectively, $C^{\alpha}(\ol\OM ; l_{\INF})$.

At the heart of the theory of $A$-analytic maps lies a Cauchy-like integral formula first introduced by Bukhgeim in \cite{bukhgeim_book}. The explicit variant \eqref{BukhgeimCauchyFormula} appeared first in Finch \cite{finch}.
The Bukhgeim-Cauchy integral formula  below is restated in terms of $\mathcal{L}$-analytic maps as opposed to $\mathcal{L}^2$-analytic as in \cite{sadiqtamasan01}.

\begin{theorem}\cite[Theorem 3.1]{sadiqtamasan01}\label{BukhgeimCauchyThm}
For some $\bg=\langle g_{-1}, g_{-2},g_{-3}...\rangle  \in l^{1,1}_{\INF}(\Gam)\cap C^\alpha(\Gam;l_1)$ define the Bukhgeim-Cauchy operator $\B $ acting on $\bg$, $$\OM\ni z\mapsto \langle (\B \bg)_{-1}(z) ,(\B \bg)_{-2}(z),(\B \bg)_{-3}(z),...\rangle,$$ by
\begin{align}
(\B \bg)_{-n}(z) &:= \frac{1}{2\pi i}\sum_{j=0}^{\infty} \int_{\Gam}
\frac{ g_{-n-j}(\zeta)\ol{(\zeta-z)}^{j}}{(\zeta-z)^{j+1}}d\zeta
\nonumber \\
 \label{BukhgeimCauchyFormula}
 & \qquad - \frac{1}{2\pi i}\sum_{j=1}^{\INF} \int_{\Gam}
\frac{ g_{-n-j}(\zeta)\ol{(\zeta-z)}^{j-1}}{(\zeta-z)^{j}}d\ol{\zeta}, \; n=1,2,3,...
\end{align}
Then $\B \bg\in C^{1,\alpha}(\OM;l_\infty)\cap C(\ol \OM;l_\infty)$ is $\mathcal{L}$-analytic.
\end{theorem}

For our purposes further regularity in $\B\bg$ will be required. Such smoothness is obtained by increasing the assumptions on the rate of decay of the terms in $\bg$ as explicit below. Let us recall the Banach space $Y_\alpha$ in \cite{sadiqtamasan01}:
\begin{equation} \label{YGamdefn}
 Y_{\alpha} = \left \{ \bg\in l^{1,2}_{\INF}(\Gamma) :  \underset{{\substack{
            \xi, \mu  \in \Gam \\
            \xi \neq \mu } }}{\sup} \sum_{j=1}^{\INF} j\frac{\lvert g_{-j}(\xi) - g_{-j}(\mu) \rvert}{|\xi - \mu|^\alpha} < \INF \right \}.
\end{equation}

\begin{proposition}\label{extraRegularity}
If $\bg\in Y_{\alpha}$, $\alpha>1/2$, then \begin{align}
\B \bg\in C^{1,\alpha}(\OM;l_1)\cap C(\ol \OM;l_1)\cap C^2(\OM;l_\infty).\end{align}
\end{proposition}
\begin{proof}
Each individual integral in \eqref{BukhgeimCauchyFormula} defines a $C^\infty(\OM)$ function. However, each differentiation essentially brings down a factor of $j$ and convergence of the resulting series need to be checked. When $\bg\in Y_\alpha$ we appeal to \cite[Corollary 4.1]{sadiqtamasan01} to conclude $\B \bg\in C^{1,\alpha}(\OM;l_1)\cap C(\ol \OM;l_1)$.

We show next that $\B \bg\in C^2(\OM;l_\infty)$. Since  we only claim  the regularity inside $\OM$, it suffices to carry the estimates in a neighborhood of a fixed point $z_0$ bounded away from the boundary.
For $z_0\in\OM$ arbitrarily fixed, let $d=\inf\{|z_0-\zeta|:\;\zeta\in\Gam\}>0$. When taking two arbitrary derivatives (formally denoted $\nabla^2$) in \eqref{BukhgeimCauchyFormula}, we can easily estimate uniformly in $\{z:\;|z-z_0|< d/2\}$:
\begin{align*}
|\nabla^2 (\B \bg)_{-n}(z)|\leq M\sup_{\zeta\in\Gam}\sum_{j=0}^\infty|g_{-n-j}(\zeta)|j^2,
\end{align*}for some $M>0$ depending only on the perimeter $|\Gam|$ of the boundary and the distance $d$, but independent of $n$. In particular,
for each $n\geq 1$,
\begin{align*}
|\nabla^2 (\B \bg)_{-n}(z)|\leq M\sup_{\zeta\in\Gam}\sum_{j=0}^\infty|g_{-n-j}(\zeta)|(n+j)^2\leq M\|\bg\|_{l^{1,2}_\infty(\Gam)}
\end{align*}

\end{proof}

The Hilbert transform defined below also accounts for some index relabelling due to the difference between $\mathcal{L}$-analytic and $\mathcal{L}^2$-analytic.
\begin{definition}
For $\bg=\langle g_{-1},g_{-2},g_{-3}...\rangle\in l^{1,1}_{\INF}(\Gam)\cap C^\alpha(\Gam;l_1)$, we define the Hilbert transform $\HT\bg$  componentwise for $n \geq 1$ by
\begin{align}\label{hilbertT}
(\HT\bg)_{-n}(\xi)&=\frac{1}{\pi} \int_{\Gam } \frac{g_{-n}(\zeta)}{\zeta - \xi} d\zeta\nonumber\\
&+\frac{1}{\pi} \int_{\Gam } \left \{ \frac{d\zeta}{\zeta-\xi}-\frac{d \ol{\zeta}}{\ol{\zeta-\xi}} \right \} \sum_{j=1}^{\infty}  g_{-n-j}(\zeta)
\left( \frac{\ol{\zeta-\xi}}{\zeta-\xi} \right) ^{j},\; \xi\in\Gam.
\end{align}
\end{definition}

The following result justifies the name of the transform $\HT$. For its proof we refer to \cite[Theorem 3.2]{sadiqtamasan01}.

\begin{theorem}\label{NecSuf}
For $0<\alpha<1$, let $\bg \in l^{1,1}_{\INF}(\Gamma)\cap C^\alpha(\Gamma;l_1)$.
For $\bg$ to be boundary value of an $\mathcal{L}$-analytic function it is necessary and sufficient that
\begin{equation} \label{NecSufEq}
 (I+i\HT) \bg=0,
 \end{equation} where $\HT$ is as in \eqref{hilbertT}.
\end{theorem}
Unlike the 0-tensor case in \cite{sadiqtamasan01}, in the case of one-tensors the even and odd modes play a different role. To emphasize this difference we separate the modes:
\begin{align}\label{uEvenOdd}
\bu^{even} := \langle u_{-2},u_{-4}...\rangle, \quad \text{and} \quad
\bu^{odd} := \langle u_{-1},u_{-3},...\rangle
\end{align} and note that if $\bu$ is $\mathcal{L}^{2}$-analytic, then $\bu^{even},\bu^{odd}$ are $\mathcal{L}$-analytic.

One of the new ingredients in our characterization is the following simple property of $\mathcal{L}$-analytic maps:

Let us consider the sequence $\{\bu^{2m-1}\}_{m\geq1}\subset C(\ol\OM;l_\INF)\cap C^1(\OM;l_\INF)$ given by
\begin{align}\label{uM}
\bu^{2m-1}:=  \langle u_{2m-1}, u_{2m-3},....,u_{1}, u_{-1},u_{-3}, u_{-5}, ... \rangle, \quad m \geq 1,
\end{align}
obtained by augmenting the sequence of negative odd indices
$\langle u_{-1},u_{-3}, u_{-5}, ... \rangle$ by $m$ many terms in the order $ \underbrace{u_{2m-1}, u_{2m-3},....,u_{1}}_{m} $.

\begin{lemma}\label{UconjLemma}
Let $\{\bu^{2m-1}\}_{m\geq1}$ be the sequence of $\mathcal{L}$-analytic maps defined in \eqref{uM}. Assume that
\begin{align}\label{uTraceConj}
u_{2m-1} \lvert_{\Gam} = \ol{u_{-(2m-1)}} \lvert_{\Gam}, \forall m\geq 1.
\end{align}
Then, for each $m \geq 1$,
\begin{align}\label{identity}
u_{2m-1}(z) = \ol{u_{-(2m-1)}}(z), \quad z \in \OM,
\end{align}
\end{lemma}

\begin{proof}
For each $m\geq 2$, since $\bu^{2m-1}$ is $\mathcal{L}$-analytic,
\begin{align*}
\ol{\del} u_{2j-1} + \del u_{2j-3} = 0, \quad 2 \leq j
\leq m.
\end{align*} By taking the conjugate and using the fact that $m$ is arbitrary, we get
\begin{align*}
\ol{\del} \ol{u_{2m-3}} + \del \ol{u_{2m-1}} = 0, \quad \forall m \geq 2.
\end{align*} In other words the sequence
$\langle \ol{u_1}, \ol{u_3}, \ol{u_5}, \cdots \rangle$ is $\mathcal{L}$-analytic, and by \eqref{uTraceConj} has the same trace as the sequence $\langle u_{-1}, u_{-3}, u_{-5}, \cdots \rangle$.
Since the latter is also $\mathcal{L}$-analytic and
$\mathcal{L}$-analytic map are uniquely determined by their traces  the identity \eqref{identity} holds.
\end{proof}

\section{Range Characterization of the Doppler Transform} \label{Range1TensorNonAtten}

In this section we establish necessary and sufficient conditions for a sufficiently smooth function $g$ on $\Gam \times \sph$ to be the Doppler data of some sufficiently smooth real valued vector field $\bF= \langle F_{1}, F_{2}\rangle$ in the sense of \eqref{Doppler_definition}, i.e.
$g$ is the trace on $\Gam\times\sph$ of some solution $u$ of
\begin{align}\label{transportEQ_a=0}
&\tta\cdot\nabla_x u(x,\tta)=\tta\cdot\bF(x), \quad x\in\OM.
\end{align}

For $z=x_1+ix_2\in \OM$, we consider the Fourier expansions of $u(z,\cdot)$ in the angular variable $\tta=\langle\cos\fii,\sin\fii\rangle$:
\begin{align*}
u(z,\theta) = \sum_{-\infty}^{\infty} u_{n}(z) e^{in\fii}.
\end{align*}Since $u$ is real valued its Fourier modes occur in conjugates,
$$u_{-n}(z)= \ol{u_{n}(z)}, \quad n \geq 0, \;z\in\OM.$$

With the Cauchy-Riemann operators defined in \eqref{CauchyRiemannOp} the advection operator becomes
\begin{align}\label{adv}
\tta\cdot\nabla = e^{-i\fii} \ol{\del} + e^{i\fii} \del.
\end{align}
Provided appropriate convergence of the series (given by smoothness in the angular variable)
we see that if $u$ solves \eqref{DopplerTransEq} then its Fourier coefficients solve the system
\begin{align}\label{NDzeroEq}
&\ol{\del} u_{1}(z) + \del u_{-1}(z) = 0, \\ \label{NAf1Eq}
&\ol{\del} u_{0}(z) + \del u_{-2}(z) = f_{1}(z), \\ \label{NDAAnalyticEq}
&\ol{\del} u_{-n}(z) + \del u_{-n-2}(z) = 0, \quad n \geq 1,
\end{align} where $f_{1} = \left( F_{1}+i F_{2} \right) /2$. The proof of Theorem \ref{NADopplerT} below also shows the converse.

The range characterization of the Doppler data $g$ will be given in terms of its Fourier modes in the angular variables:
\begin{align}\label{FourierData}
g(\zeta,\theta) = \sum_{-\infty}^{\infty} g_{n}(\zeta) e^{in\fii}, \quad \zeta \in \Gam.
\end{align}Since the trace $g$ is also real valued, its Fourier modes will satisfy
\begin{align}
g_{-n}(\zeta)= \ol{g_{n}(\zeta)}, \quad n \geq 0,\;\zeta\in\Gam.
\end{align}

From the negative even modes, we built the sequence
\begin{align}\label{gEvenNozero}
\bg^{even}:=\langle {g}_{-2}, g_{-4}, g_{-6},... \rangle.
\end{align}
For each $m \geq 1 $, we use the odd modes to build the sequence
\begin{align}\label{gM}
\bg^{2m-1}:= \langle g_{2m-1}, g_{2m-3},....,g_{1}, g_{-1},g_{-3}, g_{-5}, ... \rangle
\end{align}
by augmenting the negative odd indices by $m$-many terms in the order $$ \underbrace{g_{2m-1}, g_{2m-3},....,g_{1}}_{m}. $$

\begin{theorem}[Range characterization of Doppler Transform]\label{NADopplerT}

Let $\alpha > 1/2$.

(i) Let $\bF \in C_{0}^{1,\alpha}(\ol\OM;\BR^2)$. Then  $D \bF\cap C^{\alpha}(\Gam;C^{1,\alpha}(\sph)) \neq \emptyset $, and let $g \in D \bF\cap C^{\alpha}(\Gam;C^{1,\alpha}(\sph))$ be some Doppler data of $\bF$ in the sense of \eqref{Doppler_definition}.
 Consider the corresponding sequences $\bg^{even}$ as in
\eqref{gEvenNozero} and $\bg^{2m-1}$ for $m \geq 1$ as in \eqref{gM}. Then
 $\bg^{even}, \bg^{2m-1} \in l^{1,1}_{\INF}(\Gamma)\cap C^\alpha(\Gamma;l_1)$, for $m \geq 1$, and satisfy
\begin{align}\label{RT1TensorCond1}
&[I+i\HT]\bg^{even} =0, \\ \label{RT1TensorCond2}
&[I+i\HT]\bg^{2m-1} =0, \quad \forall m \geq 1,
\end{align} where the operator $\HT$ is the Hilbert transform in \eqref{hilbertT}.

(ii) Let $g\in C^{\alpha} \left(\Gam; C^{1,\alpha}(\sph) \right)\cap C^0(\Gam;C^{2,\alpha}(\sph))$, be real valued such that its corresponding sequence $\bg^{even}$ as in \eqref{gEvenNozero} satisfies \eqref{RT1TensorCond1} and $\bg^{2m-1}$ for $m \geq 1$, as in \eqref{gM} satisfies \eqref{RT1TensorCond2}. Then there exists a real valued vector field $\bF \in C(\OM;\BR^2)$, such that $g\in D\bF$ is the Doppler data of $\bF$.
\end{theorem}

\begin{proof} (i) For the justification of $D \bF\cap C^{\alpha}(\Gam;C^{1,\alpha}(\sph)) \neq \emptyset $, we refer to the proof of \cite[Theorem 4.1]{sadiqtamasan01} (where $\tta \cdot \bF$ is to replace the zero-tensor $f$).

Now let $g \in D \bF\cap C^{\alpha}(\Gam;C^{1,\alpha}(\sph))$ be the Doppler data of the given $\bF$. Then by
\cite[Proposition 4.1]{sadiqtamasan01} $\bg^{even}, \bg^{2m-1} \in l^{1,1}_{\INF}(\Gamma)\cap C^\alpha(\Gamma;l_1)$, for all $m \geq 1$. Moreover, $g$ is the trace on $\Gam \times \sph$ of a solution $u \in C^{\alpha}(\ol{\OM}; C^{1,\alpha}(\sph))$.

Since the Fourier modes $u_n$ of $u$ satisfies  the system  \eqref{NDAAnalyticEq} for $n$ negative even, then
$$z\mapsto  \bu^{even}(z):= \langle u_{-2}(z), u_{-4}(z), u_{-6}(z), \cdots \rangle$$
is $\mathcal{L}$-analytic in $\OM$ and the necessity part in Theorem  \ref{NecSuf} yields  \eqref{RT1TensorCond1}.

The equations \eqref{NDzeroEq} and \eqref{NDAAnalyticEq} for odd indices yield that the sequence valued map
\begin{align*}
z \mapsto \bu^{1}(z) := \langle u_{1}(z), u_{-1}(z), u_{-3}(z) \cdots \rangle
\end{align*}
is $\mathcal{L}$-analytic with the trace satisfying
\begin{align*}
u_{2k-1} \lvert_{\Gam} = g_{2k-1}, \quad k\leq 1.
\end{align*}
By the necessity part of Theorem  \ref{NecSuf}, it must be that $\bg^{1}=\langle g_{1}, g_{-1},g_{-3}, ... \rangle$ satisfies
\begin{align*}
[I+i\HT]\bg^1 =0.
\end{align*}

Recall that $u$ is real valued so that $u_{n} = \ol{u_{-n}}$ for all $n \geq 0$.
Consider now the equation \eqref{NDAAnalyticEq} for $n
= 1$ and take its conjugate to get
\begin{align}\label{delu3Eq}
\ol{\del}u_{3}+ \del u_{1}=0.
\end{align} Along with \eqref{NDzeroEq} and \eqref{NDAAnalyticEq} for odd indices, the equation \eqref{delu3Eq} yield that the sequence valued map
\begin{align*}
z \mapsto \bu^{3}(z) := \langle u_{3}(z), u_{1}(z), u_{-1}(z), u_{-3}(z) \cdots \rangle
\end{align*}
is $\mathcal{L}$-analytic with the trace satisfying
\begin{align*}
u_{2k-1} \lvert_{\Gam} = g_{2k-1}, \quad k\leq 2.
\end{align*}
By Theorem  \ref{NecSuf}, it must be that $\bg^{3}=\langle g_{3},g_{1}, g_{-1},g_{-3}, ... \rangle$ satisfies
\begin{align*}
[I+i\HT]\bg^3 =0.
\end{align*}

Inductively, the argument above holds for any odd index $2m-1$ to yield
\begin{align*}
z \mapsto \bu^{2m-1}(z) := \langle u_{2m-1}(z),u_{2m-3}(z),...,u_{1}(z), u_{-1}(z), u_{-3}(z) \cdots \rangle
\end{align*}
is $\mathcal{L}$-analytic. Then, again by the necessity part in Theorem \ref{NecSuf}, its trace $\bg^{2m-1}$ must satisfy
\begin{align*}
[I+i\HT]\bg^{2m-1} =0, \quad\forall\; m \geq 1.
\end{align*}

(ii) Assume that we have $g\in C^{\alpha} \left(\Gam; C^{1,\alpha}(\sph) \right)\cap C^0(\Gam;C^{2,\alpha}(\sph))$ with the corresponding sequences  $\bg^{even}$  satisfying \eqref{RT1TensorCond1} and $\bg^{2m-1}$  satisfying \eqref{RT1TensorCond2} for all $m \geq 1$.
We need to construct a function $u(x,\tta)$  and a real valued vector field $\bF(x)$ such that \eqref{transportEQ_a=0} is satisfied and $u|_{\Gam\times\sph}=g$. We will construct $u$ via its Fourier modes in the angular variable.

First we use the Bukhgeim-Cauchy integral formula \eqref{BukhgeimCauchyFormula} to
construct the negative even Fourier modes:
\begin{align}\label{construction_EVENS}
\langle u_{-2}, u_{-4}, u_{-6}, ..., \rangle := \B \bg^{even},
\end{align}where $\B$ is the operator in \eqref{BukhgeimCauchyFormula}.
By Theorem \ref{BukhgeimCauchyThm}, the sequence valued map
\begin{align*}
z\mapsto \langle u_{-2}(z), u_{-4}(z), u_{-6}(z), ..., \rangle ,
\end{align*}is $\mathcal{L}$-analytic in $\OM$, thus the equations
\begin{align}\label{uEvenL1}
\ol{\del} u_{-2k} + \del u_{-2k-2} = 0
\end{align} are satisfied for all $k \geq 1$.
Moreover, the hypothesis \eqref{RT1TensorCond1} and the sufficiency part of Theorem \ref{NecSuf} yields the they extend continuously to $\Gam$ and
\begin{align}
u_{-2k}|_\Gam=g_{-2k},\quad k\geq 1.
\end{align}

All of the positive even Fourier modes are constructed by conjugation:
\begin{align}\label{construct_even+}
u_{2k}:=\ol{u_{-2k}},\quad k\geq 1.
\end{align}By conjugating \eqref{uEvenL1} we note that the positive even Fourier modes also satisfy
\begin{align}\label{uEvenL1+}
\ol{\del} u_{2k+2} + \del u_{2k} = 0,\quad k\geq 1.
\end{align}Moreover, they extend continuously to $\Gam$ and
\begin{align}
u_{2k}|_\Gam=\ol{u_{-2k}}|_\Gam=\ol{g_{-2k}}=g_{2k},\quad k\geq 1.
\end{align}

Let $u_0$ be an \emph{arbitrary} function in $C^1(\OM)\cap C(\ol\OM)$ with
\begin{align}\label{u_0trace}
u_0|_{\Gam}=g_0.
\end{align}
The arbitrariness of $u_0$ characterizes the non-uniqueness (up to the gradient field of a function which vanishes at the boundary) in the reconstruction of a vector field from its Doppler data.

Next we construct the odd negative indexed modes from
$$\bg^{odd}:=\langle g_{-1}, g_{-3},...\rangle$$
via the Bukhgeim-Cauchy integral formula \eqref{BukhgeimCauchyFormula} by
\begin{align}\label{constructionODDSnegative}
\bu^{odd}= \langle u_{-1}, u_{-3},  \cdots \rangle
:= \B \bg^{odd}.
\end{align}By Theorem \ref{BukhgeimCauchyThm} the sequence $\bu^{odd}$ is $\LL$-analytic, i.e.,
\begin{align} \label{negODDS}
\ol{\del} u_{-2k-1} + \del u_{-2k-3} &= 0, \quad \forall k\geq 0,
\end{align}
Note that $\mathcal{L}\bg^1=\bg^{odd}$. By hypothesis \eqref{RT1TensorCond2}, $[I+i\HT] \bg^1 =0$. Since $\HT$ commutes with the left translation $\LL$, then
\begin{align*}
0=\LL[I+i\HT] \bg^1 =[I+i\HT]\LL \bg^1 =[I+i\HT] \bg^{odd}.
\end{align*}Applying the sufficiency part of Theorem \ref{NecSuf} we have that each $u_{2k-1}$ extends continuously to $\Gamma$ and
\begin{align}
u_{-2k-1}|_\Gam=g_{-2k-1},\quad k\geq 1.
\end{align}

So far we have constructed $z\mapsto \langle u_{-1}(z),u_{-2}(z), u_{-3}(z),...\rangle$ such that
\begin{align}\label{uL2Analy}
\ol{\del} u_{-n} + \del u_{-n-2} = 0, \\
u_{-n}|_\Gam=g_{-n},\quad \forall n\geq 1.
\end{align}

If we were to define the positive odd index modes by conjugating the negative ones (as we did for the even ones) it would not be clear why
the equation \eqref{NDzeroEq} should hold.
To solve this problem we will define the positive odd modes using the Bukhgeim-Cauchy integral formula \eqref{BukhgeimCauchyFormula} inductively.

Let $\bu^1=\langle u_{1}, u^1_{-1}, u^1_{-3},  \cdots \rangle$ be the $\mathcal{L}$-analytic map defined by
\begin{align}\label{construct_u1}
 \bu^{1}:= \B \bg^1.
\end{align}
The hypothesis \eqref{RT1TensorCond2} for $m=1$,
\begin{align*}
[I+i\HT] \bg^1 =0,
\end{align*} allows us to use the sufficiency part of Theorem \ref{NecSuf} to get that $\bu^1$ extends continuously to $\Gam$ and has trace $\bg^{1}$ on $\Gam$. However, $\LL \bu^1 = \bu^{odd}$ is also $\LL$-analytic with the same trace $\bg^{odd}$ as $\bu^{odd}$. By the uniqueness of $\LL$-analytic maps with the given trace we must have the equality
$$\langle  u^1_{-1}, u^1_{-3},  \cdots \rangle = \langle u_{-1}, u_{-3},  \cdots \rangle.$$
In other words the formula \eqref{construct_u1} constructs only one new function $u_1$ and recovers the previously defined negative odd functions $u_{-1}, u_{-3},...$. In particular $\bu^1=\langle u_{1}, u_{-1}, u_{-3},  \cdots \rangle$ is $\LL$-analytic, and
\begin{align*}
\ol{\del}u_1+ \del u_{-1} =0
\end{align*}holds in $\OM$. We stress here that, at this stage, we do not know that $u_1=\ol{u_{-1}}$.

Inductively, for $m\geq 1$, the formula
\begin{align}\label{define_u^2m-1}
\bu^{2m-1}= \langle u_{2m-1}, u^{2m-1}_{2m-3}, ..., u^{2m-1}_1, u^{2m-1}_{-1}, \cdots \rangle:=\B \bg^{2m-1}
\end{align} defines a sequence $\{\bu^{2m-1}\}_{m\geq1}$ of $\LL$-analytic maps with $\bu^{2m-1}\lvert_{\Gam} = \bg^{2m-1}$.
By the uniqueness of $\LL$-analytic maps with the given trace, a similar reasoning as above shows
\begin{align*}
\mathcal{L}\bu^{2m-1} = \bu^{2m-3}, \quad \forall m \geq 2.
\end{align*}In particular for all $m \geq 1$
\begin{align*}
\bu^{2m-1} = \langle u_{2m-1}, u_{2m-3}, ..., u_1, u_{-1}, \cdots \rangle
\end{align*} is $\LL$-analytic.

Note that the sequence $\{\bu^{2m-1}\}_{m\geq 1}$ constructed above satisfies the hypotheses of the Lemma \ref{UconjLemma}, and therefore
for each $m \geq 1$,
\begin{align}\label{exceptional}
u_{2m-1}(z) = \ol{u_{-(2m-1)}}(z), \quad z \in \OM.
\end{align}
We stress here that the identities \eqref{exceptional} need the hypothesis \eqref{RT1TensorCond2} for all $m\geq 1$, cannot be inferred directly from the Bukhgeim-Cauchy integral formula \eqref{BukhgeimCauchyFormula} for finitely many $m$'s.

We have shown that
\begin{align}\label{PostiveOdds}
\ol{\del} u_{2k-1} + \del{u_{2k-3}} = 0, \quad \forall k\in \mathbb{Z}.
\end{align}

We define now the real valued vector field $\bF = \langle 2 \re {f_1}, 2\im{f_1} \rangle$, where
\begin{align}\label{f1definition}
f_{1}:=\ol{\del}u_{0}+\del u_{-2},
\end{align}where,
by Theorem \ref{BukhgeimCauchyThm} we know that $u_{-2}\in C^{1,\alpha}(\OM)$. Together with the assumed regularity of $u_0\in C^1(\OM)$ we have   $\bF \in C(\OM;\BR^2)$.

We define $u$ by
\begin{align}\label{definitionU}
u(z, \tta):=u_{0}(z)+ \sum_{n = 1}^{\INF} u_{-n}(z)e^{-i n\fii} + \sum_{n = 1}^{\INF} u_{n}(z)e^{i n\fii}
\end{align} and check that it has the trace $g$ on $\Gam$ and satisfies the transport equation \eqref{transportEQ_a=0} in $\OM$.

Since $g\in C^{\alpha} \left(\Gam; C^{1,\alpha}(\sph) \right)\cap C^0(\Gam;C^{2,\alpha}(\sph))$, we use \cite[Corolloary 4.1]{sadiqtamasan01} and
\cite[Proposition 4.1]{sadiqtamasan01} to conclude that $u$ defined in \eqref{definitionU} belongs to $C^{1,\alpha}(\OM \times \sph)\cap C^{\alpha}(\ol{\OM}\times \sph)$. In particular $u(\cdot,\tta)$ for $\tta=\langle\cos\fii,\sin\fii\rangle$ extends to the boundary and its trace satisfies
\begin{align*}
u(\cdot,\tta)\lvert_{\Gam\times \sph} &=\left.\left ( u_{0}+ \sum_{n = 1}^{\INF} u_{-n}e^{-i n\fii} + \sum_{n = 1}^{\INF} u_{n}e^{i n\fii}\right) \right\lvert_{\Gam}\\
&=u_{0}\lvert_{\Gam}+ \sum_{n = 1}^{\INF} u_{-n}\lvert_{\Gam}e^{-i n\fii} + \sum_{n = 1}^{\INF} u_{n}\lvert_{\Gam}e^{i n\fii}\\
&=g_{0}+ \sum_{n = 1}^{\INF} g_{-n}e^{-i n\fii} + \sum_{n = 1}^{\INF} g_{n}e^{i n\fii} = g(\cdot,\tta).
\end{align*}

Since $u\in C^{1,\alpha}(\OM \times \sph)\cap C^{\alpha}(\ol{\OM}\times \sph)$, the following calculation is also justified: \begin{align*}
\tta \cdot \nabla u &= e^{-i \fii} \ol{\del}u_0 + e^{i \fii} \del u_0 + \sum_{n = 1}^{\INF} \ol{\del}u_{-n}e^{-i (n+1) \fii} + \sum_{n = 1}^{\INF} \del u_{-n}e^{-i (n-1) \fii} \\
& \qquad + \sum_{n = 1}^{\INF} \ol{\del}u_{n}e^{i (n-1) \fii} + \sum_{n = 1}^{\INF} \del u_{n}e^{i (n+1) \fii}\\
&= e^{-i \fii} (\ol{\del}u_0 +\del u_{-2}) + e^{i \fii} (\del u_0 +\ol{\del}u_2) +\ol{\del}u_1+ \del u_{-1} \\
& \qquad +\sum_{n = 1}^{\INF} (\ol{\del}u_{-n}+ \del u_{-n-2})e^{-i (n+1) \fii} + \sum_{n = 1}^{\INF} ( \ol{\del}u_{n+2}+\del u_{n}) e^{i (n+1) \fii} \\
&= e^{-i \fii} f_1 + e^{i \fii} \ol{f_1},\\
&= \tta \cdot \bF,
\end{align*}to conclude \eqref{transportEQ_a=0}.
In the third equality above we used  \eqref{uL2Analy}, \eqref{PostiveOdds}, and \eqref{f1definition}.
\end{proof}

{\bf Remark}: The Doppler data $g$ of a vector field $\bF$ and that of $\bF+ \nabla \psi$, where $\psi \in C^1(\OM)$ with $\psi \lvert_{\Gam}=0$ is the same. This non-uniqueness is transparent in our approach: the function $u_0+\psi$ satisfies
$$\ol{\del}(u_0+ \psi)+ \del{u_{-2}} = f_1+\ol\del\psi,$$ and its trace on $\Gam$ is still the same $g_0$.


\section{Range Characterization of the attenuated Doppler Transform} \label{Range1TensorAtten}

In this section we assume an attenuation $a>0$ in $\OM$ be given. We establish necessary and sufficient conditions for a sufficiently smooth function $g$ on $\Gam \times \sph$ to be the attenuated Doppler data, with attenuation $a$, of some sufficiently smooth real valued vector field $\bF= \langle F_{1}, F_{2}\rangle$ in the sense of \eqref{Doppler_definition}, i.e. $g$ is the trace on $\Gam\times\sph$ of some solution $u$ of \eqref{DopplerTransEq},
\begin{align*}
&\tta\cdot\nabla u(x,\tta)+ a(x)u(x,\tta)=\tta\cdot\bF(x), \quad x\in\OM.
\end{align*}

When $a>0$ the attenuated Doppler transform uniquely determines the vector field as shown in \cite{kazantsevBukhgeimJr04}. This uniqueness will be seen in the boundary data, where the zero-th mode $g_0$ is determined by the negative index modes.

As in \cite{sadiqtamasan01} we start by the reduction to the non-attenuated case via the special integrating factor
$e^{-h}$, where $h$ is explicitly defined in terms of $a$ by
\begin{align}\label{hDefn}
h(z,\tta) := Da(z,\theta) -\frac{1}{2} \left( I - i H \right) Ra(z\cdot \tta^{\perp},\tta),
\end{align}where $\tta^\perp$ is  orthogonal  to $\tta$, $Ra(s,\tta) = \ds \int_{-\INF}^{\INF} a\left( s \tta^{\perp} +t \tta \right)dt$ is the Radon transform of the attenuation, and
the classical Hilbert transform $H h(s) = \ds \frac{1}{\pi} \int_{-\INF}^{\INF} \frac{h(t)}{s-t}dt $ is taken in the first variable and evaluated at $s = z \cdotp \tta^{\perp}$.  The function $h$  was first considered in the work of Natterrer \cite{naterrerBook}; see also \cite{finch}, and \cite{bomanStromberg} for elegant arguments that show how $h$ extends from $\sph$ inside the disk as an analytic map.

We do not use the definition of $h$ but rather the following properties.

\begin{lemma}\label{hproperties}
Assume $a\in C^{p,\alpha}_{0}(\ol \OM)$, $p = 1,2$, $\alpha>1/2$, and $h$ defined in \eqref{hDefn}. Then $h \in C^{p,\alpha}(\ol \OM \times \sph)$ and the following hold

(i) $h$ satisfies \begin{align}\label{h=IntegratingFactor}
\tta \cdot \nabla h(z,\tta) = -a(z), \; (z, \tta) \in \OM \times \sph.
\end{align}

(ii) $h$ has vanishing negative Fourier modes yielding the expansions
\begin{align}\label{ehEq}
  e^{- h(z,\tta)} := \sum_{k=0}^{\INF} \alpha_{k}(z) e^{ik\fii}, \quad e^{h(z,\tta)} := \sum_{k=0}^{\INF} \beta_{k}(z) e^{ik\fii}, \; (z, \tta) \in \ol\OM \times \sph,
\end{align}
with

(iii)
\begin{align}
&z\mapsto \langle \alpha_{1}(z), \alpha_{2}(z), \alpha_{3}(z), ... , \rangle \in C^{p,\alpha}(\OM ; l_{1})\cap C(\ol\OM ; l_{1}), \\
&z\mapsto \langle \beta_{1}(z), \beta_{2}(z), \beta_{3}(z), ... , \rangle \in C^{p,\alpha}(\OM ; l_{1})\cap C(\ol\OM ; l_{1}).\label{betasDecay}
\end{align}

(iv) For any $z \in \OM $
\begin{align}\label{betazero}
&\ol{\del} \beta_0(z) = 0, \\ \label{betaone}
& \ol{\del} \beta_1(z) = -a(z) \beta_0(z),\\ \label{betak}
& \ol{\del} \beta_{k+2}(z) +\del \beta_{k}(z) +a(z) \beta_{k+1}(z)=0, \; k \geq 0.
\end{align}

(v) For any $z \in \OM $
\begin{align}\label{alphazero}
&\ol{\del} \alpha_0(z) = 0, \\ \label{alphaone}
& \ol{\del} \alpha_1(z) = a(z) \alpha_0(z),\\ \label{alphak}
& \ol{\del} \alpha_{k+2}(z) +\del \alpha_{k}(z) +a(z) \alpha_{k+1}(z)=0, \; k \geq 0.
\end{align}

(vi) The Fourier modes $\alpha_{k}, \beta_{k}, k\geq 0$ satisfy
\begin{align}\label{alphabetaSys}
\alpha_0 \beta_0 =1, \quad \sum_{m=0}^{k} \alpha_{m}\beta_{k-m}=0, k \geq 1.
\end{align}
\end{lemma}
\begin{proof}
The regularity of $h$ follows from \cite[Proposition 5.1]{sadiqtamasan01}. The proof of (i) follows from the fact that the term $$z\mapsto\left( I - i H \right) Ra(z\cdot \tta^{\perp},\tta)$$ is constant along lines in the direction of $\tta$.
For the proof of (ii) we refer to \cite{finch}. The regularity in (iii) follows from Bernstein's Lemma \cite[Chpt. I. Theorem 6.3]{katznelson} and the regularity of $h$.
The proof of (iv) and (v) follow by identifying like Fourier modes in the identities
$$\tta \cdot \nabla e^{\pm h(z,\tta)} = \mp e^{\pm h(z, \tta)}a(z),$$ and \eqref{adv}.
The proof of (vi) follows by identifying like Fourier modes in the identity
$e^{-h}e^{h}=1.$
\end{proof}

From \eqref{h=IntegratingFactor} it is easy to see that $u$ solves \eqref{DopplerTransEq} if and only if $v:=e^{-h}u$ solves
\begin{align}\label{transp_a>0_in_v}
\tta\cdot\nabla v(z,\tta)=(\tta\cdot\bF)e^{-h(z,\tta)}.
\end{align}

If $ u(z,\tta) = \sum_{n =-\INF}^{\INF} u_{n}(z) e^{in\fii} $ solves  \eqref{DopplerTransEq}, then its Fourier modes satisfy
\begin{align}\label{DzeroEq}
\ol{\del} u_{1}(z) &+ \del u_{-1}(z) +a(z)u_{0}(z) = 0, \\ \label{Af1Eq}
\ol{\del} u_{0}(z) &+ \del u_{-2}(z) +a(z)u_{-1}(z) = f_{1}(z), \\ \label{DAAnalyticEq}
\ol{\del} u_{n}(z) &+ \del u_{n-2}(z) +a(z)u_{n-1}(z) = 0, \quad n \leq -1,
\end{align} where $f_{1} = \left( F_{1}+i F_{2} \right) /2$  as before.

Also, if $v:=e^{-h}u=\sum_{n =-\INF}^{\INF} v_{n}(z) e^{in\fii}$ solves \eqref{transp_a>0_in_v}, then its Fourier modes satisfy
\begin{align}
\ol{\del} v_{1}(z) &+ \del v_{-1}(z) = \alpha_0(z){f_1}(z),\nonumber \\
\ol{\del} v_{0}(z) &+ \del v_{-2}(z)= \alpha_1(z){f_1}(z),  \nonumber \\
\ol{\del} v_{n}(z) &+ \del v_{n-2}(z)= 0, \quad n \leq -1,\label{DAAnalyticEq_v}
\end{align}
where $\alpha_0$, $\alpha_1$ are the Fourier modes in \eqref{ehEq}, and $f_1$  as above.

The following result shows that the equivalence between \eqref{DAAnalyticEq} and \eqref{DAAnalyticEq_v} is intrinsic to negative Fourier modes only.
\begin{lemma}\label{intrinsic} Assume $a\in C^{1,\alpha}_{0}(\ol \OM), \alpha>1/2$.

(i) Let $\bv= \langle v_{-1}, v_{-2}, ..., \rangle \in C^1(\OM, l_1)$ satisfy \eqref{DAAnalyticEq_v}, and $\bu = \langle u_{-1}, u_{-2}, ..., \rangle$ be defined componentwise by the convolution
  \begin{align}\label{UintermsV}
 u_{n} := \sum_{j=0}^{\INF} \beta_j v_{n-j}, \; n \leq -1.
 \end{align} where $\beta_j's$ are the Fourier modes in \eqref{ehEq}.
 Then $\bu$ solves \eqref{DAAnalyticEq} in $\OM$.

 (ii) Conversely, let $\bu= \langle u_{-1}, u_{-2}, ..., \rangle \in C^1(\OM, l_1)$ satisfy \eqref{DAAnalyticEq}, and $\bv = \langle v_{-1}, v_{-2}, ..., \rangle$ be defined componentwise by the convolution
  \begin{align}\label{VintermsU}
 v_{n} := \sum_{j=0}^{\INF} \alpha_j u_{n-j}, \; n \leq -1.
 \end{align} where $\alpha_j's$ are the Fourier modes in \eqref{ehEq}.
 Then $\bv$ solves \eqref{DAAnalyticEq_v} in $\OM$.
\end{lemma}
\begin{proof}
Starting from the definition, for each $n \leq -1$, we calculate
\begin{align*}
&\ol{\del}u_{n} = \sum_{j=0}^{\INF}\ol\del \beta_{j}v_{n-j} + \sum_{j=0}^{\INF}\beta_j \ol\del v_{n-j},\\
&\del u_{n-2} = \sum_{j=0}^{\INF}\del \beta_{j}v_{n-2-j} + \sum_{j=0}^{\INF}\beta_j \del v_{n-2-j}.
\end{align*}
By rearranging the terms we get
\begin{align*}
\ol{\del}u_{n} + \del u_{n-2} +& au_{n-1} = \sum_{j=0}^{\INF}\ol\del \beta_{j}v_{n-j} + \sum_{j=0}^{\INF}\del \beta_{j}v_{n-2-j} \\
&+ \sum_{j=0}^{\INF} \beta_{j} (\ol\del v_{n-j}
 + \del v_{n-2-j}) + \sum_{j=0}^{\INF} a \beta_{j} v_{n-1-j} \\
 &= \ol{\del}\beta_0 v_{n}+ (\ol{\del}\beta_1 +a\beta_0) v_{n-1} \\
 &+ \sum_{j=2}^{\INF}\left(\ol{\del}\beta_j+ \del \beta_{j-2}+ a\beta_{j-1}\right) v_{n-j} =0,
\end{align*}
where in the second equality we used \eqref{DAAnalyticEq_v} and the third equality we used the Lemma \ref{hproperties} (iv).

An analogue calculation using the properties in Lemma \ref{hproperties} (v) shows the converse.
\end{proof}

The range characterization of the attenuated Doppler data $g$ will be given in terms of its zero-th mode $g_0=\oint g(\cdot,\tta)d\tta$ and the negative index modes of
\begin{align}\label{FourierData_a>0}
e^{- h(\zeta,\tta)} g(\zeta,\theta) = \sum_{k=-\INF}^{\INF} \gamma_{k}(\zeta) e^{ik\fii}, \quad \zeta \in \Gam,
\end{align}namely,
\begin{align}\label{g_h}
\bg_h:=\langle \gamma_{-1}, \gamma_{-2}, \gamma_{-3} ... \rangle.
\end{align}
To simplify the statement, from the negative even, respectively, negative odd Fourier modes, we built the sequences
\begin{align}\label{gHEvenOdd}
\bg_h^{even} = \langle \gamma_{-2},\gamma_{-4},...\rangle, \quad \text{and} \quad
\bg_h^{odd} = \langle \gamma_{-1},\gamma_{-3},...\rangle.
\end{align}

Recall the Hilbert transform $\HT$ in \eqref{hilbertT}.
\begin{theorem}[Range characterization of attenuated Doppler Transform]\label{ADopplerT}
Let $a \in C_0^{2,\alpha}(\ol\OM)$ with  $a>0$ in $\OM$, $\alpha>1/2$.

(i) Let $\bF \in C_{0}^{1,\alpha}(\ol\OM;\BR^2)$. Then  $D_{a}\bF\cap C^{\alpha}(\Gam;C^{1,\alpha}(\sph)) \neq \emptyset $, and let $g \in D_{a}\bF\cap C^{\alpha}(\Gam;C^{1,\alpha}(\sph))$ be some Doppler data of $\bF$ in the sense of \eqref{Doppler_definition}. Consider the corresponding sequences $\bg_h^{even}, \bg_h^{odd}$ as in \eqref{gHEvenOdd}.
 Then $\bg_h^{even}, \bg_h^{odd} \in l^{1,1}_{\INF}(\Gamma)\cap C^\alpha(\Gamma;l_1)$ satisfy
\begin{align}\label{ART1TensorCond1}
&[I+i\HT] \bg_h^{even} =0, \quad [I+i\HT] \bg_h^{odd} =0.
\end{align}
Moreover, for each $\zeta\in\Gam$, the zero-th Fourier mode $g_0$ of $g$ satisfy
\begin{align}\label{ART1TensorCond2}
g_{0}(\zeta) = \lim_{\OM\ni z \to \zeta\in \Gam} \frac{-2\re  \del u_{-1}(z)}{a(z)},
\end{align} where
\begin{align*}
u_{-1}(z) = \sum_{j=0}^{\INF}\beta_{j}(z)(\B \bg_{h})_{-1-j}(z),\quad z\in\OM,
 \end{align*}
 with $\B$ be the Bukhgeim-Cauchy operator in  \eqref{BukhgeimCauchyFormula}, $\beta_{j}'s$ are the Fourier modes in \eqref{ehEq} and $\bg_h$ is the sequence of the negative Fourier modes of $e^hg$ in  \eqref{g_h}.

(ii) Let $g\in C^{\alpha} \left(\Gam; C^{1,\alpha}(\sph) \right)\cap C^0(\Gam;C^{2,\alpha}(\sph))$ be real valued, so that the corresponding sequences $\bg_h^{even}, \bg_h^{odd}\in Y_{\alpha}$. If $\bg_h^{even}, \bg_h^{odd}$ satisfy \eqref{ART1TensorCond1} and $g_0$ satisfies \eqref{ART1TensorCond2}, then there exists a unique real valued vector field $\bF \in C(\OM;\BR^2)$ such that $g\in D_{a}\bF$ is the attenuated Doppler data of $\bF$.
\end{theorem}
\begin{proof}
(i) For the justification of $D_{a}\bF\cap C^{\alpha}(\Gam;C^{1,\alpha}(\sph)) \neq \emptyset $, we refer to the proof of \cite[Theorem 5.1]{sadiqtamasan01} (where $\tta \cdot \bF$ is to replace the zero-tensor $f$).

Now let $g \in D_{a}\bF\cap C^{\alpha}(\Gam;C^{1,\alpha}(\sph))$ be the Doppler data of the given $\bF$.
Then by \cite[Proposition 4.1(i)]{sadiqtamasan01} $\bg_{h}^{even}, \bg_{h}^{odd} \in l^{1,1}_{\INF}(\Gamma)\cap C^\alpha(\Gamma;l_1)$. Moreover, $g$ is the trace on $\Gam \times \sph$ of a solution $u \in C^{1,\alpha}(\OM \times \sph)\cap C^{\alpha}(\ol{\OM}\times \sph)$ of the Transport equation \eqref{DopplerTransEq}.
Then the negative Fourier modes of $v:=e^{h}u$ satisfy \eqref{DAAnalyticEq_v}. In particular its negative odd subsequence $\langle v_{-1},v_{-3},...\rangle$ and negative even subsequence $ \langle v_{-2},v_{-4},...\rangle$ are $\LL$-analytic with traces $\bg_{h}^{odd}$ respectively $\bg_{h}^{even}$. The necessity part of Theorem \ref{NecSuf} yields \eqref{ART1TensorCond1}:
\begin{align*}
[I+i\HT]\bg_{h}^{odd}=0,\quad [I+i\HT]\bg_{h}^{even}=0.
\end{align*}

Since $a>0$ in $\OM$, and the Fourier modes $u_{1}, u_{-1}, u_{0}$ of $u$ solve \eqref{DzeroEq}, we have
\begin{align}\label{U0defn}
u_{0}(z) &= \frac{-2\re  \del u_{-1} (z)  }{a(z)}, \quad z\in \OM.
\end{align} Since the left hand side of \eqref{U0defn} is continuous all the way to the boundary, so is the right hand side. Moreover the limit below exists and
\begin{align*}
g_{0}(z_{0}) = \lim_{z \to z_{0}\in \Gam} u_{0}(z) =  \lim_{ z \to z_{0}\in \Gam} \frac{-2\re  \del u_{-1}(z)}{a(z)}
\end{align*} thus \eqref{ART1TensorCond2} holds.
This proves part (i) of the theorem. We note here that for the necessity part suffices to assume $a \in C_{0}^{1,\alpha}(\ol\OM)$.

To prove the sufficiency we will construct a real valued vector field $\bF$ in $\OM$ and a real valued function $u \in C^{1}(\OM\times \sph)\cap C(\ol\OM \times \sph)$  such that $u \lvert_{\Gam \times \sph}=g$ and $u$ solves \eqref{DopplerTransEq} in $\OM$.

Let $g\in C^{\alpha} \left(\Gam; C^{1,\alpha}(\sph) \right)\cap C^0(\Gam;C^{2,\alpha}(\sph))$, be real valued with the zero mode $g_{0}$ satisfying \eqref{ART1TensorCond2} and the corresponding sequences $\bg_h^{even}, \bg_h^{odd}$ as in \eqref{gHEvenOdd} satisfying \eqref{ART1TensorCond1}.
By \cite[Proposition 4.1(ii)]{sadiqtamasan01} and \cite[Proposition 5.2(iii)]{sadiqtamasan01} $\bg_h^{even}, \bg_h^{odd}\in Y_{\alpha}$.

Use the Bukhgeim-Cauchy Integral formula \eqref{BukhgeimCauchyFormula} to define the $\LL$-analytic maps
\begin{align}
&\bv^{even}(z)= \langle v_{-2}(z), v_{-4}(z), ... \rangle:= \B \bg_h^{even}(z), \quad z\in \OM\\
&\bv^{odd}(z)= \langle v_{-1}(z), v_{-3}(z), ... \rangle:= \B \bg_h^{odd}(z),\quad z\in \OM.
\end{align}By intertwining let also define
\[\bv:=\langle v_{-1}(z), v_{-2}(z),v_{-3}(z), ... \rangle.\]
By Proposition \ref{extraRegularity}
\begin{align}\label{smothness__v_j}
\bv^{even}, \bv^{odd},\bv\in C^{1,\alpha}(\OM; l_{1})\cap C^{\alpha}(\ol\OM;l_1)\cap C^2(\OM;l_\infty).
 \end{align}Moreover, since $\bg_h^{even}, \bg_h^{odd}$ satisfy the hypothesis \eqref{ART1TensorCond1}, by Theorem \ref{NecSuf} we have
\begin{align*}
\bv^{even} \lvert_{\Gam} = \bg_h^{even} \quad \text{and}\quad  \bv^{odd}\lvert_{\Gam} = \bg_h^{odd}.
\end{align*} In particular
\begin{align}\label{vn_intermsof_gn}
v_{n} \lvert_{\Gam} = \sum_{k=0}^{\INF} \left(\alpha_{k}\lvert_\Gam\right) g_{n-k}, \quad n \leq -1.
\end{align}

For each $n \leq -1$, we use the convolution formula below to construct
 \begin{align}\label{DefnUnConv}
 u_{n} := \sum_{j=0}^{\INF} \beta_j v_{n-j}.
 \end{align}

 Since $a\in C^{2,\alpha}_{0}(\ol\OM)$, by \eqref{betasDecay}, the sequence
 $z \mapsto \langle \beta_0(z), \beta_1(z), \beta_2(z),...\rangle$ is in $C^{2,\alpha}(\OM; l_{1}) \cap C^{\alpha}(\ol\OM ; l_1)$. Since convolution preserves $l_1$, the map is in
 \begin{align}\label{regularityofU}
 z\mapsto \langle u_{-1}(z), u_{-2}(z),...\rangle  \in C^{1,\alpha}(\OM; l_{1}) \cap C^{\alpha}(\ol\OM ; l_1).
 \end{align} Moreover, since $\bv\in C^2(\OM;l_\infty)$ as in \eqref{smothness__v_j}, we also conclude from convolution that
 \begin{align}\label{u_minus1}
 z\mapsto \langle  u_{-1}(z), u_{-2}(z),...\rangle  \in C^{2}(\OM; l_\infty).
 \end{align}

 The property \eqref{regularityofU} justifies the calculation of traces $u_n|_\Gam$
for each $n \leq -1$:
\begin{align}\nonumber
u_{n} \lvert_{\Gam} &= \sum_{j=0}^{\INF} \beta_j\left( v_{n-j} \lvert_{\Gam} \right)\\
&=\sum_{j=0}^{\INF} \beta_j \sum_{k=0}^{\INF} \alpha_{k} g_{n-j-k}= \sum_{m=0}^{\INF} \sum_{k=0}^{m}  \alpha_{k} \beta_{m-k} g_{n-m} \nonumber \\ \label{un_intermsof_gn}
&= \alpha_{0}\beta_{0}g_{n} +\sum_{m=1}^{\INF} \sum_{k=0}^{m}  \alpha_{k} \beta_{m-k} g_{n-m} =g_{n},
\end{align} where in the second equality we have used \eqref{vn_intermsof_gn}, in the third equality we  introduce the change of index $m=j+k$,  and in the last equality we used Lemma \ref{hproperties} (vi).

From the Lemma \ref{intrinsic}, the constructed $u_{n}$ in \eqref{DefnUnConv} satisfy
\begin{align}\label{uL2sys_a>0}
\ol{\del} u_{n} + \del u_{n-2} +a u_{n-1} = 0, \quad n\leq-1.
\end{align}

All of the positive Fourier modes are constructed by conjugation:
\begin{align}\label{construct_pos_Fmodes}
u_{n}:=\ol{u_{-n}},\quad n\geq 1.
\end{align}

Since $a>0$ in $\OM$ we can define $u_{0}$ by
\begin{align}\label{ConvU0defn}
u_{0}(z) := \displaystyle -\frac{2\re{\del u_{-1}(z)}}{a(z)},\quad z \in \OM,
\end{align} in particular
\begin{align}\label{u1_au_0}
\ol{\del} u_{1} + \del u_{-1} +au_{0} = 0
\end{align}
 holds.
The hypothesis \eqref{ART1TensorCond2} shows that, as defined, $u_{0}$ extends continuously to the boundary $\Gam$ and $u_{0} \lvert_{\Gam} = g_0$.
Moreover, since $u_{-1}\in C^2(\OM)$ as shown in \eqref{u_minus1} and $a\in C^2(\OM)$ we get $u_0\in C^1(\OM)$.

We next define the real valued vector field $\bF\in C(\OM;\mathbb{R}^2)$ by
\begin{align}\label{definitionF}
\bF = \langle 2 \re {f_1}, 2\im{f_1} \rangle,
\end{align}where
\begin{align}\label{f1_a>0}
f_{1}:=\ol{\del}u_{0}+\del u_{-2}+au_{-1}.
\end{align}

Finally, let $u$ be defined by
\begin{align}\label{definitionU_a>0}
u(z, \tta):=u_{0}(z)+ \sum_{n = 1}^{\INF} u_{-n}(z)e^{-i n\fii} + \sum_{n = 1}^{\INF} u_{n}(z)e^{i n\fii}
\end{align} and check that it has the trace $g$ on $\Gam$ and satisfies the transport equation \eqref{DopplerTransEq} in $\OM$.

 The regularity in \eqref{regularityofU} also allows us to take the trace as follows
\begin{align*}
u(\cdot, \tta)\lvert_{\Gam\times \sph} &=\left.\left ( u_{0}+ \sum_{n = 1}^{\INF} u_{-n}e^{-i n\fii} + \sum_{n = 1}^{\INF} u_{n}e^{i n\fii}\right) \right\lvert_{\Gam}\\
&=u_{0}\lvert_{\Gam}+ \sum_{n = 1}^{\INF} u_{-n}\lvert_{\Gam}e^{-i n\fii} + \sum_{n = 1}^{\INF} u_{n}\lvert_{\Gam}e^{i n\fii}\\
&=g_{0}+ \sum_{n = 1}^{\INF} g_{-n}e^{-i n\fii} + \sum_{n = 1}^{\INF} g_{n}e^{i n\fii} = g(\cdot, \tta),
\end{align*} where in the third equality we have used \eqref{un_intermsof_gn} and \eqref{construct_pos_Fmodes}.
Also, the regularity in \eqref{regularityofU} allows us to differentiate term-wise in order to check that $u$ satisfies \eqref{DopplerTransEq}:
\begin{align*}
\tta \cdot \nabla u &+au = e^{-i \fii} \ol{\del}u_0 + e^{i \fii} \del u_0 + \sum_{n = 1}^{\INF} \ol{\del}u_{-n}e^{-i (n+1) \fii} + \sum_{n = 1}^{\INF} \del u_{-n}e^{-i (n-1) \fii} \\
& \quad+ \sum_{n = 1}^{\INF} \ol{\del}u_{n}e^{i (n-1) \fii} + \sum_{n = 1}^{\INF} \del u_{n}e^{i (n+1) \fii} + \sum_{n = -\INF}^{\INF} a u_{n}e^{i n \fii}\\
&= e^{-i \fii} (\ol{\del}u_0 +\del u_{-2}+au_{-1})  + e^{i \fii} (\del u_0 +\ol{\del}u_2+au_{1}) \\
& \quad+ \ol{\del}u_1+ \del u_{-1}+au_{0} +\sum_{n = 1}^{\INF} (\ol{\del}u_{-n}+ \del u_{-n-2}+au_{-n-1})e^{i (n+1) \fii} \\
& \quad+ \sum_{n = 1}^{\INF} ( \ol{\del}u_{n+2}+\del u_{n}+au_{n+1}) e^{i (n+1) \fii} \\
&= e^{-i \fii} f_1 + e^{i \fii} \ol{f_1}\\
&= \tta \cdot \bF,
\end{align*}where the third equality uses   \eqref{uL2sys_a>0}, \eqref{u1_au_0}, \eqref{f1_a>0}.
\end{proof}

\section{When can the X-ray and Doppler data be mistaken for each other ?} \label{xraydopplermistaken}

By comparing the range conditions for the (non-attenuated) $X$-ray data for 0-tensors in \cite[Theorem 4.1]{sadiqtamasan01} and
for 1-tensors in Theorem \ref{NADopplerT} above, it is transparent that the two data cannot be mistaken for each other, unless they are both zero.
However, in the attenuated case with $a>0$ the situation is different. To distinguish between the data coming from the 0- and 1-tensor fields we
use the notations $g_f$, $g_\bF$  for the attenuated $X$-ray transform of the real valued function $f$, respectively, of the real valued vector field $\bF$.

In the theorem below the attenuated $X$-ray and Doppler transforms are assuming the same attenuation $a$.
\begin{theorem}\label{RadonDopplerConfound}

(i) Let $a \in C^1(\OM)\cap C(\ol\OM)$ be real valued with $a > 0 $ in $\OM$, and $f \in C^1(\OM) \cap C(\ol\OM)$ be real valued with $
\displaystyle f / a  \in C_{0}(\ol \OM)$. Then
$ \bF:=- \displaystyle \nabla \left ( \frac{f}{a} \right )$ is a real valued vector field whose attenuated Doppler data $g_\bF$ is the same as the attenuated $X$-ray data $g_f$ of $f$.

(ii) Let $a\in C^{1,\alpha}_0(\ol\OM)$ be real valued with $a>0$ in $\OM$. Assume that $\bF \in C_{0}^{1,\alpha}(\ol\OM;\mathbb{R}^2)$ is a vector field whose attenuated Doppler data $g_\bF$ equals the attenuated $X$-ray data
of some real valued $f \in C_{0}^{1,\alpha}(\ol\OM)$, $\alpha >1/2$. Then $\bF$ must be a gradient field and
$\bF =  - \displaystyle \nabla \left ( \frac{f}{a} \right )$.
\end{theorem}

\begin{proof}

(i) Assume $g_f$ is the $X$-ray data of some real valued function $f$, i.e., it is the trace on $\Gam \times \sph$ of solutions $w$ to the transport problem:
\begin{align}
\tta \cdot \nabla w +aw &= f,\label{transp_function} \\
w \lvert_{\Gam \times \sph} &= g_f. \nonumber
\end{align}

Let $\displaystyle u := w -\frac{f}{a}$ and $\bF :=  - \displaystyle \nabla \left ( \frac{f}{a} \right )$. Then
\begin{align*}
\tta \cdot \nabla u +au &= \tta\cdot\nabla\left(w-\frac{f}{a}\right)+a\left(w-\frac{f}{a}\right)\\
&=\tta\cdot \nabla w + aw-f +\tta\cdot \nabla\left(-\frac{f}{a}\right)\\
&=\tta\cdot\bF,\end{align*}
where the second equality uses \eqref{transp_function} and the definition of $\bF$. Moreover, since $f / a$ vanishes on $\Gam$, we get
\begin{align*}g_\bF=w\lvert_{\Gam \times \sph} = u\lvert_{\Gam\times \sph}+\left.\frac{f}{a}\right\lvert_{\Gam}= u\lvert_{\Gam\times \sph}=g_f.
\end{align*}

(ii) Let $\bF=\langle F_1,F_2\rangle \in C_{0}^{1,\alpha}(\ol\OM)$, $\alpha >1/2$, be a real valued vector field whose attenuated Doppler data $g_\bF$ matches the attenuated
$X$-ray data $g_f$ of some real valued function $f\in C_{0}^{1,\alpha}(\ol\OM)$, i.e. $$g_f=g_\bF=:g\in C^\alpha(\Gam; C^{1,\alpha}(\sph)).$$
Then there exist  $u,w\in C^{1,\alpha}(\OM\times\sph)\cap C^\alpha(\ol\OM\times\sph)$ solutions to the corresponding transport equations
\eqref{DopplerTransEq}, respectively \eqref{AttenRadonTEq} subject to
\begin{align*}
u \lvert_{\Gam \times \sph} = g = w \lvert_{\Gam \times \sph}.
\end{align*}
Then the corresponding sequences of non-positive Fourier modes $\{u_{-n}\}_{n\geq 0}$ of $u$ satisfy
\begin{align}
&\ol{\del}\ol{ u_{-1}} + \del u_{-1} +au_{0}=0\label{repeatU1}\\
&\ol{\del} u_{0} + \del u_{-2} +au_{-1}= \left( F_{1}+i F_{2}\right) /2,\label{repeatU0}\\
&\ol{\del} u_{-n} + \del u_{-n-2}+ au_{-n-1} = 0, \quad n \geq 1,\label{repearU-1}
\end{align}
whereas the non-positive Fourier modes $\{w_{-n}\}_{n\geq 0}$ of $w$ satisfy
\begin{align}
&\ol{\del} \ol{w_{-1}} + \del w_{-1} +a w_0 = f(z),\label{w0}\\
&\ol{\del} w_{0}(z) + \del w_{-2}(z)+a w_{-1} = 0, \label{w-0}\\
&\ol{\del} w_{-n}(z) + \del w_{-n-2}(z)+a w_{-n-1} = 0, \quad n \geq 1.\label{w-1}
\end{align}
Since the boundary data is the same $u|_{\Gam\times\sph}=w|_{\Gam\times\sph}$, we also have
\begin{align}\label{bddIdentity}
u_{-n}|_\Gam=w_{-n}|_\Gam,\quad \forall n\geq 1.
\end{align}

We claim that the systems \eqref{repearU-1} and \eqref{w-1} subject to the identity \eqref{bddIdentity}
yield
\begin{align}\label{u=w}
u_{-n}(z) = w_{-n}(z),\quad z\in\OM, \quad \forall n\geq 1.
\end{align}
Recall the integrating factor $e^{\pm h}$ with $h$ in \eqref{hDefn}. Since $a\in C^{1,\alpha}_0(\ol\OM)$, then
$e^h\in C^{1,\alpha}(\ol\OM\times\sph)$ and its Fourier modes
$\langle \beta_0,\beta_{-1},...\rangle\in C^{1,\alpha}(\OM; l_1)\cap C(\ol\OM; l_1)$ by Lemma \ref{hproperties}. Recall the
function $v:=e^{-h}u$, and introduce $\omega:=e^{-h}w$. Then their corresponding negative Fourier modes
$\langle v_{-1},v_{-2},....\rangle$ and $\langle \omega_{-1},\omega_{-2},....\rangle$
satisfy
\begin{align}\label{un=wn}
u_{-n}=\sum_{j=0}^{\INF}\beta_{j} v_{-n-j},  \quad \omega_{-n}=\sum_{j=0}^{\INF}\beta_{j} w_{-n-j}, \quad n \geq 1,
\end{align} as in Lemma \ref{intrinsic}. Moreover, $\langle v_{-1},v_{-2},....\rangle$ and $\langle \omega_{-1},\omega_{-2},....\rangle$ are $\LL^2$-analytic and coincide on $\Gam$.
By the uniqueness of $\LL^2$-analytic functions with a given trace,
they coincide inside:
\begin{align}\label{vn=wn}
v_{-n}(z)=\omega_{-n}(z), \quad z\in\OM, \quad \forall n\geq 1.
\end{align}
Using \eqref{un=wn} and \eqref{vn=wn}, we conclude for all $n \geq 1$ that
\begin{align*}
u_{-n}(z) = \sum_{j=0}^{\INF}\beta_{j}(z) v_{-n-j}(z) = \sum_{j=0}^{\INF}\beta_{j}(z) \omega_{-n-j}(z) = w_{-n}(z), \quad z \in \OM.
\end{align*}Thus \eqref{u=w} holds.

By subtracting \eqref{w-0} from \eqref{repeatU0} and using \eqref{u=w} we obtain $$\ol\del(u_0-w_0)=(F_1+iF_2)/2.$$ Since both $u_0$ and $w_0$
are real valued we see that
\begin{align}
\langle F_1, F_2\rangle=\nabla(u_0-w_0).\label{F=grad}
\end{align}

Moreover, by equation \eqref{w0},
\begin{align*}
f&=\ol\del\ol{w_{-1}}+\del w_{-1}+aw_0\\
&= \ol\del \ol {u_{-1}} + \del u_{-1}+aw_0 \\
&= \ol\del \ol {u_{-1}} + \del u_{-1}+au_0 +a(w_0-u_0)\\
&=a(w_0-u_0),
\end{align*}where the second equality uses \eqref{u=w} and the third equality uses \eqref{repeatU1}.
Therefore $u_0-w_0=-{f}/{a},$ and by  \eqref{F=grad}, $\bF = \displaystyle - \nabla \left ( \frac{f}{a} \right ).$
\end{proof}
Given the Helmholz decomposition of a vector field in gradient and a solenoidal field part, Theorem \ref{RadonDopplerConfound}(ii) yields the following.
\begin{corollary}
For a given attenuation, the attenuated Doppler data of a solenoidal compactly supported smooth real valued vector field cannot be mistaken by the attenuated $X$-ray data of a compactly supported smooth real valued function.
\end{corollary}

\section*{Acknowledgments} The work of both authors has been supported by the NSF Grant DMS-1312883.

\end{document}